\documentclass[12pt, twoside]{article}
\usepackage{times}
\usepackage{bm}
\usepackage{graphics}
\usepackage{theorem}
\usepackage{graphicx}
\usepackage{amsmath}
\usepackage{latexsym}
\usepackage{amssymb,mathrsfs}
\usepackage{natbib}
\usepackage[ruled]{algorithm2e}
\usepackage[flushmargin]{footmisc}

\theorembodyfont{\itshape}
\newtheorem{thm}{Theorem}[section]

\newtheorem{prop}{Proposition}[section]
\newtheorem{coro}{Corollary}[section]

\newtheorem{examp}{Example}[section]

{\bf}{\rm}

\theorembodyfont{\rmfamily}
{\bf}{\rm}
{\bf}{\rm}

\newtheorem{rem}{Remark}[section]{\itshape}{\rmfamily}
\newenvironment{proof}{\noindent{\it Proof.~~}}{\medskip}

\makeatletter
\def\eqnarray{\stepcounter{equation}\let\@currentlabel=\theequation
\global\@eqnswtrue
\global\@eqcnt\z@\tabskip\@centering\let\\=\@eqncr
$$\halign to \displaywidth\bgroup\@eqnsel\hskip\@centering
  $\displaystyle\tabskip\z@{##}$&\global\@eqcnt\@ne 
  \hfil$\;{##}\;$\hfil
  &\global\@eqcnt\tw@ $\displaystyle\tabskip\z@{##}$\hfil 
   \tabskip\@centering&\llap{##}\tabskip\z@\cr}
\makeatother
\makeatletter
    \renewcommand{\theequation}{%
    \thesection.\arabic{equation}}
    \@addtoreset{equation}{section}
  \makeatother

\newcommand{\vc}{\bm}

\newcommand{\ol}{\overline}

\newcommand{\wh}{\widehat}

\newcommand{\down}[2]{\smash{\lower#1\hbox{#2}}}
\newcommand{\up}[2]{\smash{\lower-#1\hbox{#2}}}

\newcommand{\dm}{\displaystyle}
\newcommand{\qed}{\hspace*{\fill}$\Box$}

\newcommand{\vmax}{\vee}
\newcommand{\vmin}{\wedge}
\newcommand{\EE}{\mathsf{E}}
\newcommand{\PP}{\mathsf{P}}

\newcommand{\bbN}{\mathbb{N}}

\newcommand{\bbS}{\mathbb{S}}

\newcommand{\bbZ}{\mathbb{Z}}

\newcommand{\rmd}{{\rm d}}

\newcommand{\deq}{\stackrel{\rmd}{=}}

\newcommand{\dd}[1]{\if#11 1\!\!1 
\else {\if#1C I\!\!\!C
\else {\if#1G I\!\!\!G 
\else {\if#1J J\!\!\!J 
\else {\if#1S S\!\!\!S
\else {\if#1Z Z\!\!\!Z
\else {\if#1Q O\!\!\!\!Q
\else I\!\!#1
\fi}
\fi}
\fi}
\fi} 
\fi} 
\fi} 
\fi} 
\makeatother

\begin{document}\thispagestyle{plain} 

\hfill

{\Large{\bf
\begin{center}
Simple perfect samplers using monotone \\birth-and-death processes%
\footnote[1]{This paper has been submitted for publication in a special issue on ``Queueing Theory and Network Applications" in Annals of Operations Research.}
\end{center}
}
}

\begin{center}
{
Hiroyuki Masuyama%
\footnote[2]{E-mail: masuyama@sys.i.kyoto-u.ac.jp}
}

\medskip

{\small
Department of Systems
Science, Graduate School of Informatics, Kyoto University\\
Kyoto 606-8501, Japan
}

\bigskip
\medskip

{\small
\textbf{Abstract}

\medskip

\begin{tabular}{p{0.85\textwidth}}
This paper proposes simple perfect samplers using monotone
birth-and-death processes (BD-processes), which draw samples from an
arbitrary finite discrete target distribution. We first construct a
monotone BD-process whose stationary distribution is equal to the
target distribution. We then derive upper bounds for the expected
coalescence time of the copies of the monotone BD-process. We also
establish upper bounds for the expected values and tail probabilities
of the running times of two perfect samplers, which are {\it Doubling
  CFTP} and {\it Read-once CFTP} using our monotone BD-process. The
latter sampler can draw samples {\it exactly} from unnormalized
target distributions with little memory consumption.
\end{tabular}
}
\end{center}

\begin{center}
\begin{tabular}{p{0.90\textwidth}}
{\small

\medskip

}
\end{tabular}

\end{center}

\section{Introduction}\label{sec-intro}

Perfect sampling algorithms are based on ``Coupling From The Past
(CFTP)", proposed by \break \citet{Prop96}. CFTP is a powerful technique that
enables us to perform {\it perfect sampling} from the target
distribution, i.e., to generate, in a finite time, samples that {\it
  perfectly} follow the target distribution. Basically, CFTP is time-
and memory-consuming because we have to check whether or not the
copies of a Markov chain used for CFTP coalesce at a single state
every time we extend the sample paths of the copies to the past.

\citet{Prop96} stated that CFTP is effectively achieved by a monotone
Markov chain (see, e.g., \citealt{Keil77}) constructed from the target
distribution, which is called {\it monotone CFTP or monotonic CFTP
  (MCFTP)}. As far as we know, there have been a small number of
examples for which MCFTP algorithms are established, for example,
attractive spin systems (\citealt{Prop96}), closed Jackson networks
(\citealt{Kiji08-SIAM,Kiji08-ANNOR}), discretized Dirichlet
distributions (\citealt{Mats10}) and truncated Gaussian distributions
(\citealt{Phil03}). In particular, the algorithms proposed by
\citet{Kiji08-SIAM,Kiji08-ANNOR} and \citet{Mats10} are remarkably
fast, though they are somewhat {\it sophisticated}.

The main purpose of this paper is to establish simple perfect
samplers, which draw samples from an arbitrary target distribution
$\{\pi(i);i \in \bbS\}$ on an arbitrary finite discrete set $\bbS$. It
should be noted that $\bbS$ is mapped one-to-one to a finite set of
nonnegative numbers. Thus, we assume, without loss of generality, that
$\bbS = \{0,1,\dots,N\} =: \bbZ_N$, where $N$ is a positive
integer. We also assume that
\begin{equation}
\min_{i\in\bbZ_N}\pi(i)>0.
\label{cond-pi}
\end{equation}
For later use, let $\bbN = \{1,2,3,\dots\}$, $\bbZ_+ =
\{0,1,2,\dots\}$, $\bbZ = \{0,\pm1,\pm2,\dots\}$ and $\bbZ_n =
\{0,1,\dots,n\}$ for any $n \in \bbZ_+$. For $n,m\in\bbZ$ such that $n
\le m$, let $\bbZ_{[n,m]} = \{n,n+1,\dots,m-1,m\}$. Let $x \vmax y =
\max(x,y)$ and $x \vmin y = \min(x,y)$ for $x,y\in
(-\infty,\infty)$. Furthermore, we use the notation $f(x) = O(g(x))$
to represent $\limsup_{x\to\infty}|f(x)|/|g(x)| < \infty$.

In this paper, we first construct a monotone birth-and-death process
(monotone BD-process or MBD for short) whose stationary distribution
is equal to the target distribution $\{\pi(i);i\in\bbZ_N\}$. More
specifically, we construct a monotone stochastic
matrix $\vc{P}:=(P(i,j))_{i,j\in\bbZ_N}$ such that
\begin{align}
\vc{P}
&=
\left(
\begin{array}{cccccc}
r_0 	& 
p_0 	& 
0 		& 
 		& 
\cdots	&
0 		
\\
q_1 	& 
r_1 	& 
p_1 	& 
0		& 
\cdots	&
0 		
\\
0    	& 
q_2		&
r_2 	& 
p_2		&
\ddots	&
\vdots 				
\\
\vdots	& 
\ddots	&  
\ddots  & 
\ddots  & 
\ddots  & 
0
\\
0 		& 
\cdots	&
0		&
q_{N-1} &
r_{N-1} &
p_{N-1} 
\\
0 		& 
\cdots 	& 
\cdots	&
0 		& 
q_{N} &
r_N 	
\end{array}
\right),
\label{defn-P}
\end{align}
where 
\begin{align}
p_i
&=
\left\{
\begin{array}{ll}
\dm{ 1 \over 1 + \gamma(0)}, &  \quad i = 0,
\\
\dm{ 1 \over 1 + \gamma(i) \vmax \gamma(i-1)}, &  \quad i  \in \bbZ_{[1,N-1]},
\end{array}
\right.
\label{defn-p_i}
\\
q_i
&= 
\left\{
\begin{array}{ll}
\dm{ \gamma(0) \over 1 + \gamma(0)}, &  \quad i = 1,
\\
\dm{ \gamma(i-1) \over 1 + \gamma(i-1) \vmax \gamma(i-2)}, &  
\quad i \in \bbZ_{[2,N]},
\end{array}
\right.
\label{defn-q_i}
\\
\gamma(i) 
&= {\pi(i) \over \pi(i+1)},\qquad
i \in \bbZ_{N-1}.
\label{defn-gamma(i)}
\end{align}
By definition, $r_i = 1 - p_i - q_i$ for $i \in \bbZ_N$ and $q_0 = p_N
= 0$. We prove that $\vc{P}$ is an irreducible and monotone stochastic
matrix whose stationary distribution is equal to the target
distribution $\{\pi(i);i\in\bbZ_N\}$ (see Theorem~\ref{thm-monotone}
below). We then discuss the first time when the copies of the MBD
characterized by $\vc{P}$ coalesce at a single state, which is called
the {\it coalescence time} and denoted by $T_{\rm C}$. Utilizing the
existing results on BD-processes, we derive the upper bound for the expected
coalescence time:
\begin{equation}
\EE[T_{\rm C}]  \le \theta N,
\label{ineqn-E[T_C]-00}
\end{equation}
where $\theta \in (0,\infty)$ is a certain parameter (possibly
depending on $N$).

Next we consider Doubling CFTP and Read-once CFTP (see, e.g.,
\citealt{Hube2016-book}) using our MBD, which is referred to as {\it
  Doubling-MBD sampler} and {\it Read-once-MBD sampler}, respectively.
Using (\ref{ineqn-E[T_C]-00}), we obtain upper bounds for the expected
values and tail probabilities of the running times of Doubling-MBD and
Read-once-MBD samplers. These upper bounds show that the expected
running times of the two MBD samplers are $O(\theta N)$, and thus they
are slower than the sophisticated special-purpose
algorithms mentioned above. However, the construction of our MBD is
very simple and little memory-consuming. In general, Doubling MCFTP
and Read-once MCFTP are easily implementable (for details, see, e.g.,
\citealt{Hube2016-book}). Therefore, Doubling-MBD and Read-once-MBD
samplers are easily implementable and general-purpose perfect sampling
algorithm. Furthermore, Read-once-MBD sampler is little
memory-consuming, though the sampler is somewhat more time-consuming
than Doubling-MBD sampler.  As a result, Read-once-MBD sampler can
draw samples from unnormalized target distributions with little
memory consumption. This is a remarkable feature of Read-once-MBD
sampler.

The rest of this paper is divided into two
sections. Section~\ref{sec-monotone-BD} discusses our MBD constructed
from the target distribution. Section~\ref{sec-doubling-CFTP}
considers the performance of the two perfect samplers using our MBD.

\section{Monotone BD-process from the target distribution}\label{sec-monotone-BD}

This section consists of two
subsections. Section~\ref{subsec-construction-BD-process} constructs a
monotone BD-process (MBD) whose stationary distribution is equal to
the target distribution. Section~\ref{subsec-mean-coalescence-time}
derives some upper bounds for the expected coalescence time of the
copies of the MBD.

\subsection{Construction of a monotone BD-process from the target distribution}\label{subsec-construction-BD-process}

The following theorem is the fundamental result of this paper.
\begin{thm}\label{thm-monotone}
The stochastic matrix $\vc{P}$ defined by (\ref{defn-P}) together with
(\ref{defn-p_i})--(\ref{defn-gamma(i)}) is an irreducible and monotone
one whose stationary distribution is equal to the target distribution
$\{\pi(i);i\in\bbZ_N\}$.
\end{thm}

\begin{proof}
From (\ref{cond-pi}) and (\ref{defn-p_i})--(\ref{defn-gamma(i)}), we have
\begin{align}
&&&&
p_i &> 0, & i &\in \bbZ_{N-1},&&&&
\nonumber
\\
&&&&
q_i &> 0,  & i &\in \bbZ_{[1,N]},&&&&
\nonumber
\\
&&&&
\pi(i) q_i &= \pi(i-1)p_{i-1}, & i &\in \bbZ_{[1,N]},&&&&
\label{balance-eq}
\end{align}
which show that $\vc{P}$ is an irreducible stochastic matrix and that
the target distribution $\{\pi(i);i\in\bbZ_N\}$ is a reversible
measure and thus a unique stationary distribution of $\vc{P}$.
Therefore, it suffices to prove that
\begin{align}
&&&&
p_i &\le 1 - q_{i+1},
& i & \in \bbZ_{N-1}, &&&&
\label{ineq-monotone}
\\
&&&&
p_i &\le 1 - q_i, & i & \in \bbZ_{N-1},&&&&
\nonumber
\end{align}
where (\ref{ineq-monotone}) is the condition for the monotonicity of
$\vc{P}$ (see, e.g., \citealt[Definition 1.2]{Keil77}).

From (\ref{defn-p_i}) and (\ref{defn-gamma(i)}), we have, for $i\in
\bbZ_{N-1}$,
\begin{equation}
p_i
\le 
{1 \over 1 + \gamma(i)}
= { \pi(i+1) \over \pi(i+1) + \pi(i) },
\label{eqn-170217-01}
\end{equation}
and thus
\begin{eqnarray*}
p_i
&\le& 1 - { \pi(i) \over \pi(i+1) + \pi(i) }
= 1 - { \pi(i) \over \pi(i+1) } { \pi(i+1) \over \pi(i+1) + \pi(i) }
\nonumber
\\ 
&\le& 1 - { \pi(i) \over \pi(i+1) } p_i
= 1 - q_{i+1},
\end{eqnarray*}
where the last inequality follows from (\ref{eqn-170217-01}) and the
last equality follows from (\ref{balance-eq}). Similarly, for $i\in
\bbZ_{N-1}$,
\begin{eqnarray*}
p_i
&\le& {1 \over 1 + \gamma(i-1)}
= { \pi(i) \over \pi(i) + \pi(i-1) }
\nonumber
\\ 
&=& 1 - { \pi(i-1) \over \pi(i) + \pi(i-1) }
= 1 - { \pi(i-1) \over \pi(i) } { \pi(i) \over \pi(i) + \pi(i-1) }
\nonumber
\\ 
&\le& 1 - { \pi(i-1) \over \pi(i) } p_{i-1}
= 1 - q_i.
\end{eqnarray*}
The proof is completed. \qed
\end{proof}

\medskip

The following corollary is immediate from Theorem~\ref{thm-monotone}.
\begin{coro}\label{coro-monotone}
Suppose that the conditions of Theorem~\ref{thm-monotone} are satisfied. 
We then have the following:
\begin{enumerate}
\item If $\{\gamma(i);i \in \bbZ_{N-1}\}$ is nondecreasing, i.e.,
\[
\gamma(0) \le \gamma (1) \le \cdots \le \gamma(N-1),
\]
then (\ref{defn-p_i}) and (\ref{defn-q_i}) are reduced to
\begin{align}
&&&&
p_i
&=  {1 \over 1 + \gamma(i)},
& i &\in \bbZ_{N-1}, &&&&
\label{defn-p_i-02}
\\
&&&& q_i &=  {\gamma(i-1) \over 1 + \gamma(i-1)},
& i &\in \bbZ_{[1,N]}. &&&&
\label{defn-q_i-02}
\end{align}
\item If $\{\gamma(i);i \in \bbZ_{N-1}\}$ is nonincreasing, i.e.,
\begin{equation}
\gamma(0) \ge \gamma (1) \ge \cdots \ge \gamma(N-1),
\label{ratio-nonincreasing}
\end{equation}
then (\ref{defn-p_i}) and (\ref{defn-q_i}) are reduced to
\begin{align*}
p_i
&=  
\left\{
\begin{array}{ll}
\dm{1 \over 1 + \gamma(0)}, & \quad i = 0,
\\
\rule{0mm}{7mm}
\dm{1 \over 1 + \gamma(i-1)}, & \quad i \in \bbZ_{[1,N-1]},
\end{array}
\right.
\\
q_i &=  
\left\{
\begin{array}{ll}
\dm{\gamma(0) \over 1 + \gamma(0)}, & \quad i = 1,
\\
\rule{0mm}{7mm}
\dm{\gamma(i-1) \over 1 + \gamma(i-2)}, & \quad i \in \bbZ_{[2,N]}.
\end{array}
\right.
\end{align*}
\end{enumerate}

\end{coro}

\begin{rem}
Suppose that the conditions of the statement (i) of Corollary~\ref{coro-monotone} are satisfied.
Let $\{\wh{Y}_n;n\in\bbZ_+\}$ denote an MBD with state space $\bbZ_N$ and transition probability matrix $\vc{P}$ in (\ref{defn-P}) together with (\ref{defn-p_i-02}) and (\ref{defn-q_i-02}). Furthermore, suppose that the BD-processes $\{\wh{Y}_n\}$ starts with an initial distribution $\{\wh{\pi}(i);i\in\bbZ_N\}$ such that
\[
{\wh{\pi}(0) \over \pi(0)} \ge
{\wh{\pi}(1) \over \pi(1)} \ge \cdots \ge
{\wh{\pi}(N) \over \pi(N)}.
\]
Note here that (\ref{ratio-nonincreasing}) yields
\[
\pi(i-1) \pi(i+1) \le [ \pi(i) ]^2, \qquad i\in \bbZ_{[1,N-1]},
\]
which shows that the target distribution $\{\pi(i)\}$ is
log-concave. Therefore, it follows from \citet[Proposition~3.2,
  Corollary~3.3(a) and Theorem~5.1]{Fill13} that the BD-process
$\{\wh{Y}_n\}$ mixes (i.e., converges to stationarity) faster in total
variation distance than does an arbitrary MBD
$\{\wh{Z}_n;n\in\bbZ_+\}$ that has the same state space $\bbZ_N$,
stationary distribution $\{\pi(i);i\in\bbZ_N\}$ and initial
distribution $\{\wh{\pi}(i);i\in\bbZ_N\}$ as the BD-process
$\{\wh{Y}_n;n\in\bbZ_+\}$.
\end{rem}

Next we describe a construction of the copies of the MBD with state
space $\bbZ_N$ and transition probability matrix $\vc{P}$, which can
be used for MCFTP. To this end, we define $\{U_m;m \in \bbZ\}$ as a
sequence of independent and identically distributed (i.i.d.) uniform
random variables in $(0,1)$. We then have the following result.
\begin{thm}\label{thm-copy-MC}
Suppose that the conditions of Theorem~\ref{thm-monotone} are
satisfied.  Let $\phi:\bbZ_N \times (0,1) \to \bbZ_N$ denote a
function such that, for $i \in \bbZ_N$ and $u \in (0,1)$,
\begin{equation}
\phi(i,u)
=
\left\{
\begin{array}{ll}
i+1, & \quad u \in (1 - p_i,1),
\\
i,   & \quad u \in [q_i,1 - p_i],
\\
i-1, & \quad u \in (0,q_i),
\end{array}
\right.
\label{defn-phi(i,u)}
\end{equation}
where $p_i$ and $q_i$ are given in (\ref{defn-p_i}) and
(\ref{defn-q_i}), respectively.  Furthermore, for each $k \in \bbZ_N$,
let $\{X_n^{(k)};n\in\bbZ_+\}$ denote a sequence of random variables
such that
\begin{equation}
X_n^{(k)}
=
\left\{
\begin{array}{ll}
k, & \quad n=0,
\\
\phi(X_{n-1}^{(k)},U_n), & \quad n \in \bbN.
\end{array}
\right.
\label{defn-X_n^{(k)}}
\end{equation}
Under these conditions, the stochastic processes
$\{X_n^{(k)};n\in\bbZ_+\}$'s, $k \in \bbZ_N$, are MBDs with transition
probability matrix $\vc{P}$, which satisfy
\begin{equation}
X_n^{(0)} \le X_n^{(1)} \le \cdots \le X_n^{(N)}\quad
\mbox{for all $n \in \bbN$}.
\label{pathwise-ordering}
\end{equation}

\end{thm}

\begin{proof}
It is clear that $\{X_n^{(k)};n\in\bbZ_+\}$'s, $k \in \bbZ_N$, are
MBDs with transition probability matrix $\vc{P}$. Thus, we prove that
(\ref{pathwise-ordering}) holds.

It follows from (\ref{defn-phi(i,u)}) that, for $i\in \bbZ_{N-1}$,
\begin{align*}
&&&&
\phi(i+1,u) &\ge i+1 \ge \phi(i,u), &  1 > u &\ge q_{i+1}, &&&&
\\
&&&&
\phi(i+1,u) &\ge i \ge \phi(i,u), &   0  < u &\le 1 - p_i. &&&&
\end{align*}
It also follows from (\ref{ineq-monotone}) that $q_{i+1} \le 1 - p_i$
for $i\in \bbZ_{N-1}$. Therefore,
\begin{equation}
\phi(i+1,u) \ge \phi(i,u),\qquad i\in \bbZ_{N-1},\ u \in (0,1).
\label{ineqn-phi(i,u)}
\end{equation}
Combining (\ref{defn-X_n^{(k)}}) and (\ref{ineqn-phi(i,u)}) yields
(\ref{pathwise-ordering}). \qed
\end{proof}

Theorem~\ref{thm-copy-MC} shows that the function $\phi$, together
with the uniform random variables $U_m$'s, generates MBDs with
transition probability matrix $\vc{P}$. Thus, we refer to $\phi$ as a
{\it monotone update function} for MBDs with $\vc{P}$. Note here that
$\{X_n^{(k)};n\in\bbZ_+\}$'s can be considered the copies of a generic
BD-process driven by the monotone update function $\phi$, which is
denoted by $\{X_n;n\in\bbZ_+\}$.  Especially, we refer to
$\{X_n^{(N)}\}$ and $\{X_n^{(0)}\}$ as the {\it upper-bounding} and
{\it lower-bounding} copies, respectively, of $\{X_n\}$.

\subsection{Expected coalescence time of the copies of the monotone BD-process}\label{subsec-mean-coalescence-time}

Let $T_{\rm C}$ denote
\begin{equation}
T_{\rm C} = \inf\{n\in\bbN: X_n^{(0)} = X_n^{(1)} = \cdots = X_n^{(N)}\},
\label{defn-T_C}
\end{equation}
which is the first time when all the copies $\{X_n^{(k)}\}$'s coalesce
at a single state in the state space $\bbZ_N$. Thus, we call $T_{\rm
  C}$ the {\it coalescence time} of the copies $\{X_n^{(k)}\}$'s of
$\{X_n\}$.  It follows from (\ref{pathwise-ordering}) and
(\ref{defn-T_C}) that
\begin{equation}
T_{\rm C} = \inf\{n\in\bbN: X_n^{(0)} = X_n^{(N)}\}.
\label{eqn-T_C}
\end{equation}

We now define $T_{i,j}$, $i,j \in \bbZ_N^2$, $i \neq j$, as a generic
random variable for the first passage time from state $i$ to state
$j$. We assume that $\{T_{0,1},T_{1,2},\dots,T_{N-1,N}\}$ are
independent and so are $\{T_{N,N-1},T_{N-1,N-2},\dots,T_{1,0}\}$,
which does not lose generality due to the skip-free property of
BD-processes. It then follows from (\ref{eqn-T_C}) that
\begin{eqnarray*}
T_{\rm C} &\le& \inf\{n\in\bbN: X_n^{(0)} = N\} \deq T_{0,N}
\deq \sum_{i=0}^{N-1} T_{i,i+1},
\label{ineqn-T_C-01}
\\
T_{\rm C} &\le& \inf\{n\in\bbN: X_n^{(N)} = 0\} \deq T_{N,0}
\deq \sum_{i=0}^{N-1} T_{i+1,i} ,
\label{ineqn-T_C-02}
\end{eqnarray*}
where the symbol ``$\deq$" represents the equality in distribution. 
Therefore, 
\begin{eqnarray}
\EE[T_{\rm C}] 
&\le& \EE[T_{0,N}] \vmin \EE[T_{N,0}].
\label{ineqn-E[T_C]}
\end{eqnarray}
We can readily obtain (see, e.g., Theorem 4.11 of \citet{Heym04},
where continuous-time BD-processes are considered)
\begin{eqnarray}
\EE[T_{0,N}]
&=& \sum_{i=0}^{N-1} \EE[T_{i,i+1}]
= \sum_{i = 0}^{N-1} {1 \over p_i}
 \sum_{m=0}^i {  \pi(m) \over \pi(i) }.
\label{eqn-E[T_{0,N}]}
\\
\EE[T_{N,0}]
&=& \sum_{i=0}^{N-1} \EE[T_{i+1,i}]
= \sum_{i = 0}^{N-1} {1 \over p_i}
 \sum_{m=i+1}^N { \pi(m) \over \pi(i) }.
\label{eqn-E[T_{N,0}]}
\end{eqnarray}
Substituting (\ref{eqn-E[T_{0,N}]}) and (\ref{eqn-E[T_{N,0}]}) into
(\ref{ineqn-E[T_C]}) yields
\begin{equation}
\EE[T_{\rm C}] 
\le \left(
\sum_{i = 0}^{N-1} {1 \over p_i}
\sum_{m=0}^i {  \pi(m) \over \pi(i) }
\right)
\vmin 
\left(
\sum_{i = 0}^{N-1} {1 \over p_i}
\sum_{m=i+1}^N { \pi(m) \over \pi(i) }
\right).
\label{ineqn-mean-T_C}
\end{equation}
Using (\ref{ineqn-mean-T_C}), we obtain the following result.
\begin{thm}\label{thm-mean-T_C}
If the conditions of Theorem~\ref{thm-monotone} are satisfied, then
\begin{equation}
\EE[T_{\rm C}] 
\le \theta N,
\label{bound-mean-T_C}
\end{equation}
where $\theta$ is a positive constant such that
\begin{equation}
\theta
=  
\left[
\max_{i\in\bbZ_{N-1}}
\left( {1 \over p_i} \sum_{m=0}^i { \pi(m) \over \pi(i) } \right)
\right]
\vmin
\left[
\max_{i\in\bbZ_{N-1}}
\left( {1 \over p_i} \sum_{m=i+1}^N { \pi(m) \over \pi(i) } \right)
\right].
\label{defn-theta}
\end{equation}
\end{thm}

\begin{proof}
Note that
\begin{eqnarray*}
\sum_{i = 0}^{N-1} {1 \over p_i}
\sum_{m=0}^i {  \pi(m) \over \pi(i) }
&\le&
 \max_{i\in\bbZ_{N-1}}
\left( {1 \over p_i} \sum_{m=0}^i { \pi(m) \over \pi(i) } \right)N,
\\
\sum_{i = 0}^{N-1} {1 \over p_i}
\sum_{m=i+1}^N { \pi(m) \over \pi(i) }
&\le& \max_{i\in\bbZ_{N-1}}
\left( {1 \over p_i} \sum_{m=i+1}^N { \pi(m) \over \pi(i) } \right)N.
\end{eqnarray*}
Substituting these inequalities into (\ref{ineqn-mean-T_C}) leads to
(\ref{bound-mean-T_C}) with (\ref{defn-theta}). \qed
\end{proof}

\medskip

Under some additional conditions, we obtain simpler bounds for $\EE[T_{\rm C}]$.
\begin{thm}\label{thm-mean-T_C-02}
Suppose that the conditions of Theorem~\ref{thm-monotone} are satisfied. 
We then have the following:
\begin{enumerate}
\item If there exists some $C \in (0,\infty)$ independent of $N$ such that
\begin{equation}
\max_{i \in \bbZ_{N-1}} \sum_{m=0}^i {  \pi(m) \over \pi(i) } \le C
\quad \mbox{{\rm or}} \quad
\max_{i \in \bbZ_{N-1}} \sum_{m=i+1}^N {  \pi(m) \over \pi(i) } \le C,
\label{cond-pi-geo}
\end{equation}
then
\begin{equation}
\EE[T_{\rm C}] \le   
C\left(\max_{i\in\bbZ_{N-1}} {1 \over p_i} \right) N.
\label{bound-mean-T_C-geo}
\end{equation}
\item If there exists some $C \in (0,\infty)$ independent of $N$ such that
\begin{equation}
\max_{i \in \bbZ_{N-1}} \max_{0 \le m \le i} {  \pi(m) \over \pi(i) } \le C
\quad \mbox{{\rm or}} \quad
\max_{i \in \bbZ_{N-1}} \max_{i+1 \le m \le N} {  \pi(m) \over \pi(i) } \le C,
\label{cond-pi-long-tailed}
\end{equation}
then
\begin{equation}
\EE[T_{\rm C}] \le   
C\left(\max_{i\in\bbZ_{N-1}} {1 \over p_i} \right) {N(N+1)\over 2}.
\label{bound-mean-T_C-long-tailed}
\end{equation}

\end{enumerate}

\end{thm}

\begin{rem}\label{rem-pi-monotone}
If $\{\pi(i);i \in \bbZ_N\}$ is nonincreasing or nondecreasing, then
(\ref{cond-pi-long-tailed}) holds for $C=1$ and thus the statement
(ii) of Theorem~\ref{thm-mean-T_C-02} yields
\begin{equation}
\EE[T_{\rm C}] \le   
\left(\max_{i\in\bbZ_{N-1}} {1 \over p_i} \right){N(N+1)\over 2}.
\label{bound-mean-T_C-02}
\end{equation}
\end{rem}

\medskip

\noindent
{\it Proof of Theorem~\ref{thm-mean-T_C-02}.~~} We first prove the
statement (i). Applying (\ref{cond-pi-geo}) to (\ref{ineqn-mean-T_C})
yields
\begin{eqnarray*}
\EE[T_{\rm C}] 
\le C \sum_{i = 0}^{N-1} {1 \over p_i}
\le C\left(\max_{i\in\bbZ_{N-1}} {1 \over p_i} \right) N,
\end{eqnarray*}
which shows that (\ref{bound-mean-T_C-geo}) holds.  Next we prove the
statement (ii).  Combining (\ref{cond-pi-long-tailed}) and
(\ref{ineqn-mean-T_C}), we have either of the following inequalities:
\begin{eqnarray*}
\EE[T_{\rm C}] 
&\le& 
\sum_{i = 0}^{N-1} {1 \over p_i} \sum_{m=0}^i C
\le
C\left(\max_{i\in\bbZ_{N-1}} {1 \over p_i} \right)
\sum_{i = 0}^{N-1} (i+1),
\end{eqnarray*}
\begin{eqnarray*}
\EE[T_{\rm C}] 
&\le& \sum_{i = 0}^{N-1} {1 \over p_i} \sum_{m=i+1}^N C
\le
C\left(\max_{i\in\bbZ_{N-1}} {1 \over p_i} \right)
\sum_{i = 0}^{N-1} (N-i).
\end{eqnarray*}
Each of the two inequalities shows that
(\ref{bound-mean-T_C-long-tailed}) holds.  \qed

\begin{examp}[Truncated geometric distribution]
Consider a truncated geometric distribution. To this end, fix
\begin{equation*}
\pi(i) = 
{(1- \xi)\xi^i \over 1 - \xi^{N+1}}, \qquad i\in \bbZ_N,
\label{eqn-pi(k)-geo}
\end{equation*}
where $0 < \xi < 1$. Clearly, $\gamma(i) = \xi^{-1}$ for $i\in
\bbZ_{N-1}$, which satisfies the conditions of the statement (i) of
Corollary~\ref{coro-monotone}.  Thus, from (\ref{defn-p_i-02}) and
(\ref{defn-q_i-02}), we have
\[
p_i
= {1 \over 1 + \gamma(i)}
= {\xi \over 1 + \xi}, \qquad i \in \bbZ_{N-1}. 
\]
Note here that
\begin{eqnarray*}
\max_{i \in \bbZ_{N-1}} \sum_{m=0}^i { \pi(m) \over \pi(i) } 
&=&
\max_{i \in \bbZ_{N-1}} \sum_{m=0}^i \xi^{m-i}
\le  {1 \over 1 - \xi},
\\
\max_{i \in \bbZ_{N-1}} \sum_{m=i+1}^N { \pi(m) \over \pi(i) } 
&=&
\max_{i \in \bbZ_{N-1}} \sum_{m=i+1}^N \xi^{m-i}
\le  {\xi \over 1 - \xi}.
\end{eqnarray*}
Combining these results and the statement (i) of
Theorem~\ref{thm-mean-T_C-02} yields
\begin{equation*}
\EE[T_{\rm C}] \le 
{\xi \over 1 - \xi}{1 + \xi \over \xi }N
= {1 + \xi \over 1 - \xi} N.
\end{equation*}
\end{examp}

\begin{examp}[Zipf distribution]\label{examp-Zipf}
Consider the following Zipf distribution $\{\pi(i);i\in\bbZ_N\}$:
\begin{equation*}
\pi(i) = {(i+1)^{-\alpha} \over \sum_{\ell=0}^N (\ell+1)^{-\alpha} },
\qquad  i\in \bbZ_N,
\end{equation*}
where $\alpha > 1$. We then have
\[
\gamma(i) = \left(1 + {1 \over i+1} \right)^{\alpha},
\qquad i\in \bbZ_{N-1},
\]
which is decreasing with $i$. Therefore, according to the statement
(ii) of Corollary~\ref{coro-monotone}, we fix
\begin{equation}
p_i 
=
{1 \over 1 + \gamma((i-1) \vmax 0)}
=
{ (i \vmax 1)^{\alpha} 
\over 
(i \vmax 1)^{\alpha} + \{  (i \vmax 1) + 1 \}^{\alpha}  }, 
\qquad i \in \bbZ_{N-1}.
\label{eqn-p_i-Pareto}
\end{equation}
Furthermore, since $\{\pi(i);i\in\bbZ_N\}$ is decreasing (see
Remark~\ref{rem-pi-monotone}), it follows from
(\ref{bound-mean-T_C-02}) and (\ref{eqn-p_i-Pareto}) that
\begin{eqnarray*}
\EE[T_{\rm C}]
&\le& 
\max_{i \in \bbZ_{[1,N-1]}} 
\left[
{ i^{\alpha} + (i+1)^{\alpha} \over i^{\alpha}  }
\right]
{N(N+1)\over 2}
\le (1 + 2^{\alpha}){N(N+1)\over 2}.
\end{eqnarray*}
\end{examp}

\section{Perfect samplers using the monotone BD-process}\label{sec-doubling-CFTP}

In this section, we discuss the running times of Doubling CFTP and
Read-once CFTP using the monotone update function $\phi$, which are
referred to as {\it Doubling-MBD sampler} and {\it Read-once-MBD
  sampler}, respectively.

To facilitate the subsequent discussion, we introduce some
definitions.  For $m \in \bbZ$ and $n \in \bbZ_+$, let
\[
\vc{U}_m^{(n)} = (U_m,U_{m+1},\break\dots,U_{m+n-1}).
\]
For
convenience, let $\vc{U}_m^{(-n)} = \emptyset$ for all $m \in \bbZ$
and $n \in \bbN$. In addition, for $s \in \bbZ$, $n \in \bbZ_+$ and $x
\in \bbZ_N$, let $\varPhi_s^{s+n}(x,\vc{U}_m^{(n)})$ denote
\begin{equation*}
\varPhi_s^{s+n}(x,\vc{U}_m^{(n)})
= \phi (\phi( \cdots \phi(x,U_m),\dots,U_{m+n-2}),U_{m+n-1}),
\end{equation*}
where $\phi$ is the monotone update function given in (\ref{defn-phi(i,u)}) and
$\{U_m;m \in \bbZ\}$ is a sequence of i.i.d.\ uniform random variables in $(0,1)$.
Note that, for any $t \in \bbZ_+$, the two processes
$\{\varPhi_{-t}^{-t+n}(N,\vc{U}_{-t}^{(n)});n\in \bbZ_+\}$ and
$\{\varPhi_{-t}^{-t+n}(0,\vc{U}_{-t}^{(n)});n\in \bbZ_+\}$ are
the upper- and lower-bounding copies of an MBD with
transition probability matrix $\vc{P}$, which run from time $-t$ to time $-t+n$.

We first consider Doubling MBD sampler, which is described in
Algorithm~\ref{algo-doubling-CFTP}.

%
\begin{algorithm}[H]
\caption{Doubling-MBD sampler}\label{algo-doubling-CFTP}
{
{\bf Output:} $X$
\begin{enumerate}
\item Set $t=2$.
\item Double $t$ until 
\[
Y:=\varPhi_{-t}^{-t/2}(0,\vc{U}_{-t}^{(t/2)}) = 
\varPhi_{-t}^{-t/2}(N,\vc{U}_{-t}^{(t/2)}).
\]
\item Return $X=\varPhi_{-t/2}^0(Y,\vc{U}_{-t/2}^{(t/2)})$.
\end{enumerate}
}
\end{algorithm}

Let $T_{\rm D}$ denote the number of the uniform random variables used
by Algorithm~\ref{algo-doubling-CFTP}, i.e., $T_{\rm D}$ is equal to a
positive integers such that
\begin{equation*}
T_{\rm D}
= \inf\{t \in \bbN: 
\varPhi_{-t}^{-t/2}(0,\vc{U}_{-t}^{(t/2)}) = 
\varPhi_{-t}^{-t/2}(N,\vc{U}_{-t}^{(t/2)})\}.
\end{equation*}
Following \citet{Hube08}, we read $T_{\rm D}$ as the running time of
Algorithm~\ref{algo-doubling-CFTP}. Using \citet[Lemma 5.4]{Hube08},
we obtain the following result.
\begin{prop}[Doubling-MBD sampler]\label{prop-doubling-MCFTP}
\begin{eqnarray}
\EE[T_{\rm D}] &\le& 4\theta N,
\label{ineqn-E[T_D]}
\\
\PP(T_{\rm D} > k\theta N)
&\le& \exp\{1 - k/(4e)\},\qquad k \in \bbZ_+,
\label{ineqn-P(T_D>k)}
\end{eqnarray}
where $\theta$ is the positive constant given in (\ref{defn-theta}).
\end{prop}

\begin{proof}
It follows from \citet[Lemma 5.4]{Hube08} that
\begin{eqnarray*}
\EE[T_{\rm D}] &\le& 4\EE[T_{\rm C}],
\\
\PP(T_{\rm D} > k\EE[T_{\rm C}])
&\le& \exp\{1 - k/(4e)\},\qquad k \in \bbZ_+.
\end{eqnarray*}
Combining these with Theorem~\ref{thm-mean-T_C} results in
(\ref{ineqn-E[T_D]}) and (\ref{ineqn-P(T_D>k)}). \qed
\end{proof}

Next we consider Read-once-MBD sampler, which is described in
Algorithm~\ref{algo-read-once} below.

%
\begin{algorithm}[H]
\caption{Read-once-MBD sampler}\label{algo-read-once}
{
{\bf Input:} Block size $B \in \bbN$ \\
{\bf Output:} $X$
\begin{enumerate}
\item Set $\ell=1$.
\item If
\[
\varPhi_0^B(0,\vc{U}_{(\ell-1) B}^{(B)}) \neq 
\varPhi_0^B(N,\vc{U}_{(\ell-1) B }^{(B)}),
\]
then increment $\ell$ by one and go back to Step (ii); otherwise set
$X = \varPhi_0^B(0,\vc{U}_{(\ell-1) B}^{(B)})$ and go to Step
(iii) with $\ell' = 1$.
\item Set $Y=X$ and perform the following: If
\[
\varPhi_{B + (\ell'-1)B}^{B + \ell'B}(0,\vc{U}_{\ell B + (\ell'-1) B}^{(B)}) 
\neq 
\varPhi_{B + (\ell'-1)B}^{B + \ell'B}(N,\vc{U}_{\ell B + (\ell'-1) B}^{(B)}),
\]
then set $X = \varPhi_{B + (\ell'-1)B}^{B + \ell'B}(Y,\vc{U}_{\ell B +
  (\ell'-1) B}^{(B)}) $ and go back to Step (iii) with incrementing
 $\ell'$ by one; otherwise return $X$.
\end{enumerate}
}
\end{algorithm}

\begin{rem}
When Algorithm~\ref{algo-read-once} stops, we have 
\[
X = \varPhi_0^{\ell'B}(0,\vc{U}_{(\ell-1) B}^{\ell' B})
  = \varPhi_0^{\ell'B}(N,\vc{U}_{(\ell-1) B}^{\ell' B}).
\]
\end{rem}

As with Algorithm~\ref{algo-doubling-CFTP}, we define $T_{\rm R}$ as
the number of the uniform random variables used by
Algorithm~\ref{algo-read-once}, and then read $T_{\rm R}$ as the
running time of Algorithm~\ref{algo-read-once}. Let $L$ and $L'$
denote the numbers of the iterations in Steps (ii) and (iii),
respectively, of Algorithm~\ref{algo-read-once}. By definition, $L$
and $L'$ are independent and
\begin{equation}
\PP(L > k) = \PP(L' > k) = [ \PP(T_{\rm C} >  B) ]^k,
\qquad k \in \bbZ_+.
\label{eqn-P(L>k)-01}
\end{equation}
In addition, 
\begin{equation}
T_{\rm R} = (L + L')B.
\label{eqn-T_R}
\end{equation}
Using Theorem~\ref{thm-mean-T_C} together with (\ref{eqn-P(L>k)-01})
and (\ref{eqn-T_R}), and proceeding as in the proof of 
\citet[Lemma~5.4]{Hube08}, we obtain the following result.
\begin{thm}[Read-once-MBD sampler]\label{thm-read-once-MCFTP}
Fix $b \in \bbN$ such that $b > e$, and fix the block size $B$ of
Algorithm~\ref{algo-read-once} such that $B = b \lceil \theta \rceil
N$, where $\theta$ is the positive constant given in
(\ref{defn-theta}). We then have
\begin{eqnarray}
\EE[T_{\rm R}] &\le& { 2b \lceil \theta \rceil N \over 1 - \beta(b) },
\label{ineqn-E[T_R]}
\\
\PP(T_{\rm R} > b \lceil \theta \rceil N k)
&\le& (1 - \beta(b)) [\beta(b)]^{k-1}k + [\beta(b)]^k,\qquad k \in \bbN,
\label{ineqn-P(T_R>k)}
\end{eqnarray}
where $\beta(b) = \exp\{1 - b/e\} \in (0,1)$.  In addition, the value
of integer $b > e$ minimizing the right hand side of
(\ref{ineqn-E[T_R]}) is equal to six, or equivalently,
\begin{equation}
\arg\min_{b\in \{3,4,5,\dots\}}
{ 2b \lceil \theta \rceil N \over 1 - \beta(b) } = 6.
\label{eqn-argmin}
\end{equation}
\end{thm}

\begin{proof}
%
It follows from Markov's inequality that, for any fixed $\alpha > 1$, 
\begin{equation}
\PP(T_{\rm C} > \alpha \EE[T_{\rm C}]) \le  \alpha^{-1}.
\label{ineqn-P(T_C>alpha*E[T_C])}
\end{equation}
Note here that $\{\PP(T_{\rm C} > x); x \ge 0\}$ is log-subadditive
(see, e.g., \citealt[Theorem~6]{Prop96}). Thus, from
(\ref{ineqn-P(T_C>alpha*E[T_C])}), we have
\begin{eqnarray}
\PP(T_{\rm C} > k\EE[T_{\rm C}])
&\le& \left( {1 \over \alpha} \right)^{\lfloor k/\alpha \rfloor}
\le \alpha \left( {1 \over \alpha} \right)^{k/\alpha}
\nonumber
\\
&=& \alpha \exp\{- k (\ln \alpha) / \alpha\},\qquad k \in \bbZ_+.
\label{ineqn-P(T_C>alpha*E[T_C])-01}
\end{eqnarray}
We now fix $\alpha = e$ to maximizing $(\ln \alpha) / \alpha$. It then
follows from (\ref{ineqn-P(T_C>alpha*E[T_C])-01}) that
\begin{eqnarray}
\PP(T_{\rm C} > k\EE[T_{\rm C}])
&\le& \exp\{1 - k/e \} = \beta(k),\qquad k \in \bbZ_+.
\label{ineqn-P(T_C>alpha*E[T_C])-02}
\end{eqnarray} 
From $B = b \lceil \theta \rceil N$ and Theorem~\ref{thm-mean-T_C}, we
also have $B \ge b\EE[T_{\rm C}]$. Using this and
(\ref{ineqn-P(T_C>alpha*E[T_C])-02}), we obtain
\begin{eqnarray}
\PP(T_{\rm C} > B)
&\le& \PP(T_{\rm C} > b \EE[T_{\rm C}])
\le \beta(b).
\label{ineqn-P(T_C>B)-01}
\end{eqnarray} 
Substituting (\ref{ineqn-P(T_C>B)-01}) into (\ref{eqn-P(L>k)-01}) yields
\begin{equation*}
\PP(L > k) = \PP(L' > k) \le [\beta(b)]^k,
\qquad k \in \bbZ_+.
\end{equation*}
Therefore, there exist independent random variables $\ol{L}$ and
$\ol{L}'$ such that
\begin{eqnarray}
L &\le& \ol{L},  \qquad L' \le \ol{L}',
\label{ineqn-L-L'}
\\
\PP(\ol{L} > k) &=& \PP(\ol{L}' > k) = [\beta(b)]^k,
\qquad k \in \bbZ_+.
\label{prob-ol{L}}
\end{eqnarray}
It follows from (\ref{prob-ol{L}}) that
\begin{equation*}
\EE[L] + \EE[L'] = { 2 \over 1 - \beta(b) },
\end{equation*}
and
\begin{eqnarray*}
\PP(\ol{L} + \ol{L}' > k) 
&=& \sum_{m=1}^k \PP(\ol{L}=m)\PP(\ol{L}' > k-m) + \PP(\ol{L} > k)
\nonumber
\\
&=&  (1 - \beta(b))
\sum_{m=1}^k [\beta(b)]^{m-1} [\beta(b)]^{k-m} + [\beta(b)]^k
\nonumber
\\
&=&  (1 - \beta(b)) [\beta(b)]^{k-1} k
 + [\beta(b)]^k,\qquad k \in \bbN.
\end{eqnarray*}
Combining these results with (\ref{eqn-T_R}) and (\ref{ineqn-L-L'}),
we have
\begin{eqnarray*}
\EE[T_{\rm R}] 
&\le& B \cdot \EE[ \ol{L} + \ol{L}'] = { 2B \over 1 - \beta(b) },
\\
\PP(T_{\rm R} > kB)
&\le& \PP(\ol{L} + \ol{L}' > k) 
= (1 - \beta(b)) [\beta(b)]^{k-1} k
 + [\beta(b)]^k,\quad k \in \bbN,
\end{eqnarray*}
which imply that (\ref{ineqn-E[T_R]}) and (\ref{ineqn-P(T_R>k)}) hold
due to $B = b \lceil \theta \rceil N$.

In what follows, we prove (\ref{eqn-argmin}), which is equivalent to
\begin{equation}
\arg\min_{x \in \{3,4,5,\dots\}} F(x) = 6,
\label{eqn-argmin_F(x)}
\end{equation}
where $F$ denotes a function such on $(e,\infty)$ that
\[
F(x) = {x \over 1 - \beta(x)}
= {x \over  1 - \exp\{1 - x/e\}},\qquad x > e.
\]
By definition, $F$ is convex and
\[
F'(x) = 
{ 1 - \exp\{1 - x/e\} - e^{-1}x \exp\{1 - x/e\} 
\over  [1 - \exp\{1 - x/e\}]^2 
},\qquad x > e.
\]
Let $G(x)$, $x > e$, denote the numerator of $F'(x)$ in the above
equation, i.e.,
\[
G(x) = 1 - \exp\{1 - x/e\} - e^{-1}x \exp\{1 - x/e\},\qquad x > e.
\]
We then have
\begin{eqnarray*}
G(2e) &=&  
1 - 3e^{-1} < 0,
\\
G(2.5e) &=& 
1 - 3.5 (e \sqrt{e})^{-1} 
> 1 - 3.5 \times (2.5 \times 1.5)^{-1} 
= {1 \over 15}>0,
\end{eqnarray*}
which lead to $F'(2e) < 0$ and $F'(2.5e) > 0$. Note here that $2e >
5.4$ and $2.5e < 7$. Therefore, the convexity of $F$ yields $F'(5) <
0$ and $F'(7) > 0$, which results in (\ref{eqn-argmin_F(x)}). \qed
\end{proof}

Proposition~\ref{prop-doubling-MCFTP} and
Theorem~\ref{thm-read-once-MCFTP} imply that the running time $T_{\rm
  D}$ of Doubling-MBD sampler is less than the running time $T_{\rm
  R}$ of Read-once-MBD sampler. However, Doubling-MBD sampler has to
store all the generated (uniform) random numbers until it outputs a
sample following the target distribution. On the other hand,
Read-once-MBD sampler is little memory-consuming because the sampler
uses, only one time, each of the generated random numbers.

We close this section by comparing our perfect samplers with the
inverse transform sampling (see, e.g., \citealt{Fish96}). The inverse
transform sampling for discrete target distributions is easy
implementable and takes the $O(N)$ running time in order to draw a
sample from the target distribution. Therefore, the inverse transform
sampling is less time-consuming than our perfect samplers.

To discuss this topic from a different perspective, we suppose that the
target distribution $\{\pi(i);i\in\bbZ_N\}$ is not normalized, in
other words, we have an unnormalized target distribution
$\{\wh{\pi}(i);i\in\bbZ_N\}$ such that
$C_{\pi}:=\sum_{i=0}^N\wh{\pi}(i) \neq 1$ and
\begin{equation}
\pi(i) = {1 \over C_{\pi}}\wh{\pi}(i),\qquad i \in \bbZ_N.
\label{eqn-wh{pi}(i)}
\end{equation}
It then follows from (\ref{defn-gamma(i)}) and (\ref{eqn-wh{pi}(i)}) that
\[
\gamma(i) = {\wh{\pi}(i) \over \wh{\pi}(i+1)},\qquad
i \in \bbZ_{N-1}.
\]
Therefore, our two perfect samplers still work well by using the
unnormalized target distribution $\{\wh{\pi}(i)\}$. On the other hand,
the inverse transform sampling has a problem in the present situation
because it needs the cumulative distribution
$\{\sigma(i);i\in\bbZ_N\}$, where $\sigma(i) = \sum_{\ell=0}^i
\pi(\ell)$ for $i\in\bbZ_N$. To obtain the cumulative distribution
$\{\sigma(i)\}$, we have to compute the normalizing constant $C_{\pi}$
by summing the unnormalized target distribution $\{\wh{\pi}(i)\}$ over
its support set $\bbZ_N$.

It should be note that the obtained constant $C_{\pi}$
includes, at worst, the $O(N)$ rounding error. Such rounding error can
be reduced to $O(\ln N)$ if $C_{\pi}$ is computed by pairwise
summation (see, e.g., \citealt{High93}). Furthermore, if $C_{\pi}$ is
computed by Kahan summation algorithm, then the rounding error can be
basically reduced to $O(1)$ but its computational complexity is four
times as much as that of {\it naive} summation (see, e.g.,
\citealt{High93}). Even though we take any of these options, we have to
store all the information of the cumulative distribution
$\{\sigma(i)\}$. Such memory consumption is not necessary for our two
perfect samplers.

As a result, although our MBD perfect samplers may not be particularly
superior in speed to other methods, they are easily implementable and
can draw samples {\it exactly} from unnormalized target
distributions. Especially, Read-one MBD sampler achieves such {\it
  exact} sampling with little memory consumption.

\section*{Acknowledgments}
The author thanks Dr.~Shuji Kijima for helpful comments that motivated this work.
%
%
\bibliographystyle{elsarticle-harv} 
%
%


\end{document}